\documentclass[11pt]{amsart}
\usepackage{amsmath}
\usepackage{amsfonts,amssymb}
\usepackage{url}
\usepackage{epsfig,latexsym,graphicx}
\usepackage{esint}
\usepackage{graphicx,latexsym,amssymb,amsmath,amsfonts,fancyhdr,mathrsfs}
\usepackage{euscript}
\usepackage[letterpaper,hmargin=1.2in,vmargin=1.0in]{geometry}
\usepackage{saltscribe}
\usepackage{caption}
\usepackage{subcaption}
\usepackage{hyperref}
\usepackage{tikz}
\usetikzlibrary{arrows}
\usetikzlibrary{matrix}
\usetikzlibrary{decorations.markings}

\newtheorem{theorem}{Theorem}[section]
\newtheorem{lemma}[theorem]{Lemma}

\numberwithin{equation}{section}

\newtheorem{remark}{Remark}

\newcommand{\bfo}{\mathbf{o}}

\newcommand{\bfn}{\mathbf{n}}
\newcommand{\bfr}{\mathbf{r}}
\newcommand{\TF}{\mathrm{TF}}
\newcommand{\ce}[1]{Ce^{-c{#1}^3}}
\title[Local Transversal Fluctuations in LPP]{ Lower bound for large local transversal fluctuations of Geodesics in Last Passage Percolation}
\author{Pranay Agarwal}
\address{Pranay Agarwal, Department of Mathematics, Indian Institue of Technology, Kanpur, India}
\email{pranayag@iitk.ac.in}

\begin{document}

\maketitle

\begin{abstract}
    For exactly solvable models of planar last passage percolation, it is known that geodesics of length $n$ exhibit transversal fluctuations at scale $n^{2/3}$ and matching (up to exponents) upper and lower bounds for the tail probabilities are available. The local transversal fluctuations near the endpoints are expected to be much smaller; it is known that the transversal fluctuation up to distance $r \ll n$ is typically of the order $r^{2/3}$ and the probability that the fluctuation is larger than $tr^{2/3}$ is at most $\ce{t}$. In this note, we provide a short argument establishing a matching lower bound for this probability.
\end{abstract}

\section{Introduction}
We consider the Exponential Last Passage Percolation model on the integer lattice $\bZ^2$, where one assigns independently sampled values from the exponential distribution to each vertex of the lattice. We study up-right paths from the origin to the point $(n,n)$ in this model, the paths with the maximal weight sum are called the geodesic $\Gamma_n$, which are almost surely unique in this model. It was proved in \cite{Johan} that $\Gamma_n$ shows transversal fluctuations of the order $n^{2/3 + o(1)}$. Later the result was improved in \cite[Theorem 11.1]{Basu2014Aug} and subsequently in \cite[Proposition C.9]{Basu2021May}, which showed that the geodesic $\Gamma_n$ exhibits transversal fluctuations at scale $n^{2/3}$, and the probability that these fluctuations are greater than $tn^{2/3}$ is at most $\ce{t}$. One can establish a matching lower bound from the arguments of \cite[Proposition 1.4]{hammond2018modulus}, even though their arguments focused on the Poissonian LPP model. As one would expect, the transversal fluctuations are much smaller near the endpoints of the geodesic. At a distance $r$ from the endpoints, where $r \ll n$, the fluctuations are typically of the order $r^{2/3}$.

Now we define transversal fluctuations formally as
$$\TF^n(r) :=\max \bc{ \abs{x-y} : (x,y) \in \Gamma_n, 0 
\leq x+y \leq 2r}.$$
Thus, $\TF^n(n)$ is the global transversal fluctuation of $\Gamma_n$ and $\TF^n(r)$ are the local transversal fluctuations. It was shown in \cite[Proposition 2.1]{Balazs2023Aug} that for large $t$, for some constants $c,C,$
    $$\P{\TF^n(r) > tr^{2/3}} \leq C\exp\br{-ct^3}.$$
We will give a more precise statement later as Theorem \ref{transfl}. Our main result gives a matching lower bound for this probability.
 \begin{theorem}
      \label{nproof}
      There exist positive constants $c_1, C_1, r_0, t_0$ such that for all $\epsilon >0$, $r > r_0$, $n \geq r$ and $(1-\epsilon) r^{1/3} > t> t_0$
    $$\P{\TF^n(r) > tr^{2/3}} \geq C_1\exp\br{-c_1t^3}.$$
 \end{theorem}
This completes the picture for local transversal fluctuations of finite geodesics. The main tools we use to prove the bound are moderate deviation estimates from integrable probability along with the ordering of geodesics. We shall simplify the arguments presented in \cite[Proposition 1.4]{hammond2018modulus} for Exponential LPP and use them to prove Theorem \ref{thm2}. The resulting construction that we get from Theorem \ref{thm2} will then be used to prove Theorem \ref{nproof}.

Some related expressions have appeared in the literature before as well. A similar lower bound for the transversal fluctuation of semi-infinite geodesics was shown in \cite[Theorem 2.8]{Timo} by utilizing stationary exponential LPP. A formula to calculate the probability of the geodesic to pass through a given point was derived in \cite[Theorem 1.1]{liu2022onepoint}. However, as remarked by the author, the complicated expression of their formula makes it seem intangible to use their result to derive any useful asymptotics.

\section{Notation and Preliminaries}
For a point $u \in \bZ^2$, we use $X_u$ to denote the exponential random variable associated with $u$, also known as the passage time at $u$. We use $\mathbf{o}$ to denote the origin and $\mathbf{n}$ to denote the point $(n,n)$ for $n \in \bN$. We represent the line $x+y = T$ by $\cL_T$. Let $\gamma$ be an up-right path from $u$ to $v$, we define passage time of the path $\gamma$ as
$$\ell(\gamma) = \sum_{w \in \gamma \setminus \bc{v}} X_w.$$
We will use $T(u,v)$ to denote the maximal passage time over all up-right paths from $u$ to $v$, also known as the last passage time form $u$ to $v$. $\Gamma_u$ denotes the path achieving maximal length from $\bfo$ to $u$. As a special case, we will use $\Gamma_n$ to denote the geodesic from $\bfo$ to $\bfn$. We define $f: \bR^2 \times \bR^2 \to \bR$ as 
$$f(u,v) = \begin{cases}
    \br{\sqrt{v_1- u_1} + \sqrt{v_2 - u_2}}^2& \text{if } u \preceq  v,\\
    0 , &\text{otherwise},
\end{cases} $$
where $u \preceq v$ means that $u$ is coordinate wise smaller than $v$. Similarly, we define the functions, 
\begin{align*}
    \phi(u) = u_2+ u_1 && \mathrm{and,} && \psi(u) = u_2 - u_1.
\end{align*}
We use $C, c$ to denote arbitrary constants which may change from equation to equation. We have avoided using floor and ceiling functions in many places to make notations less cumbersome. Adding the floor/ceiling functions will not change our arguments in a non-trivial way.

\subsection{Estimates for temporal fluctuations}  Johansson proved \cite[Remark 1.5]{Jour} that the last passage time is equal in distribution to the top eigenvalue of the Laguerre Unitary Ensemble, and upper bounds on the upper and lower tails on this eigenvalue were proved in \cite[Theorem 2]{LR}. Their arguments gave us matching bounds for the upper tail and an upper bound for the lower tail. A matching lower bound for the lower tail was then proved in \cite[Theorem 2]{Basu2021Apr}. We summarize the results in the following theorem.
\begin{theorem}
\label{tempfl}
    Let $n \in \bN, h \in (0,1]$. For $u =(u_1, u_2) \in \bZ^2$ such that $\phi(u) = n$ and $h \leq \frac{u_1}{u_2} \leq 1/h$. Then there exist positive constants $c_1, C_1, c_2, C_2, c_3, C_3, c_4, C_4, n_0, t_0$ depending on $h$ such that for $n > n_0$ and $t_0 < t< n^{2/3}$,
    \begin{equation}
    \label{1}
    C_1\exp\br{-c_1 t^{3/2}} \leq \P{T(\bfo, u) - f(\bfo, u) > tn^{1/3}} \leq C_2\exp\br{-c_2 t^{3/2}},
    \end{equation}
    \begin{equation}    
    \label{2}
    C_3\exp\br{-c_3 t^{3}} \leq \P{T(\bfo, u) - f(\bfo, u) \leq -tn^{1/3}} \leq C_4\exp\br{-c_4 t^{3}}.
    \end{equation}
    
\end{theorem}

\subsection{Estimates for transversal fluctuations} Exponential decay for the upper tail of global transversal fluctuation was shown in \cite[Theorem 11.1]{Basu2014Aug}. Later, the optimal bound was found in \cite[Proposition C.9]{Basu2021May}. Upper bound for large local transversal fluctuations for $r \ll n$ was shown in \cite[Theorem 3]{Basu2019Sep} and was improved in \cite[Proposition 2.1]{Balazs2023Aug} . We summarize their results as follows.
\begin{theorem}
    \label{transfl}
    There exist positive constants $c_1, C_1, r_0, n_0, t_0$ such that for $r> r_0, n> r \vee n_0$ and $t > t_0$,
    $$\P{\TF^n(r) > tr^{2/3}} \leq C_1 \exp\br{-c_1 t^3}.$$
\end{theorem} 

\section{Proof of Main Result}
 To motivate our arguments, we provide a proof for the global transversal fluctuation of the geodesic $\Gamma_r$. The arguments essentially follows from \cite[Proposition 1.4]{hammond2018modulus}, see Remark \ref{rmk: t2}.
 \begin{theorem}
    There exist constants $r_0, t_0, c_1, C_1$ such that for all $\epsilon >0$, $r > r_0$ and $(1-\epsilon)r^{1/3} > t > t_0$,
	\label{thm2}
	$$\P{\TF^r(r) \geq tr^{2/3}} \geq C_1e^{-c_1t^3}.$$
\end{theorem}
Consider a rectangle in our integer lattice formed by the lines $x-y = 0$, $x+y = 0$, $y = x + r^{2/3}$ and $x + y = r/3$. Let the midpoints of the smaller sides be denoted by $u,v$. Let $A$ be the event that a path $\gamma$ from $u$ to $v$ lies inside this rectangle and has a passage time greater than $4r + 10t^2 r^{1/3}$.
\begin{lemma}
	\label{lemma6}
	There exist constants $c,C, r_0, t_0$ such that for $r> r_0$ and $t>t_0$,
	$$\P{A} \geq Ce^{-ct^3}.$$ 
\end{lemma}
\begin{remark}
    Our proof is a combination of arguments that are known in the literature. They can be found in \cite[Lemma 12.1]{Basu2014Aug} and \cite[Theorem 2]{Ganguly2023Jun}. For completeness, we provide a short self-contained proof.
\end{remark}
\begin{proof}
\begin{figure}
        \centering
        \includegraphics[scale = 0.8]{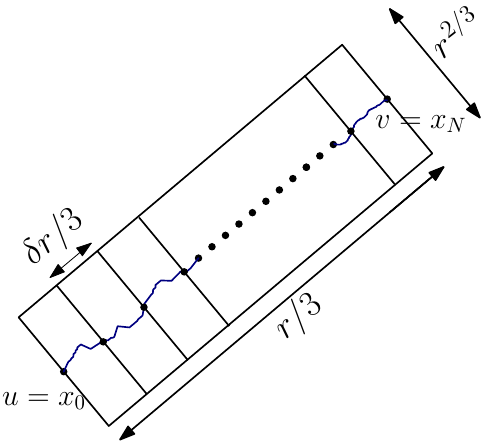}
        \caption{The proof of Lemma \ref{lemma6} is illustrated here; the blue paths indicate a path with large passage time between the midpoints that lie inside the rectangle. Connecting these blue paths gives us a long path between the points $x_0$ and $x_N$.}
        \label{fig:1}
    \end{figure}
	Let $\delta >0$, we shall later choose $\delta$ depending on $t$. We partition the given rectangle into $N:= \delta^{-1}$ pieces using equidistant lines parallel to the smaller side; without loss of generality we assume that $\delta^{-1}$ is a natural number. We denote the midpoint of the $i^{th}$ line by $x_i$ and define $x_0 = u$ and $x_{N} =v$. Refer to Figure \ref{fig:1} for an illustration of the setup.
 
	Let $A_i$ be the event that there is a path lying inside the rectangle from $x_{i-1}$ to $x_i$ which has a passage time greater than $4\delta r + 10t\br{\delta r}^{1/3}$. By Theorem \ref{transfl} and Theorem \ref{tempfl}, we have constants $c_1, c_2, C_1, C_2$ such that
	\begin{align*}
	\P{A_i} &\geq C_1e^{-c_1t^{3/2}} - C_2e^{-c_2(\delta^{-2/3})^3}\\
	&\geq C_1e^{-c_1t^{3/2}} - C_2e^{-c_2(\delta^{-1})}.
	\end{align*}
	For the right-hand side to be positive, we require $\delta^{-1} \geq c_0t^{3/2}$ for a suitable choice of $c_0$. Thus we choose $\delta^{-1} = c_0t^{3/2}$ and assume $t$ to be large enough. This allows us to write
	$$\P{A_i} \geq Ce^{-c t^{3/2}} \geq e^{-c t^{3/2}}.$$
    Now, since all the events $A_i$ are independent of each other, we have that
	$$\P{\bigcap_{1 \leq i \leq N} A_i} \geq \br{e^{-c t^{3/2}}}^{\delta^{-1}} \geq e^{-c t^{3}}.$$
	But if the event $A_i$ holds for all $i\in \bc{1,2,3 \dots N}$, then the concatenation of these paths will have passage time greater than $ 4r + 10t \delta^{-2/3} r^{1/3} = 4r + 10ct^2r^{1/3}$.
	Thus, we have proved our claim that,
	$$\P{A} \geq Ce^{-c t^{3}}.$$
\end{proof}

\begin{proof}[Proof of Theorem \ref{thm2}]
    \begin{figure}
        \centering
        \includegraphics[scale = 0.9]{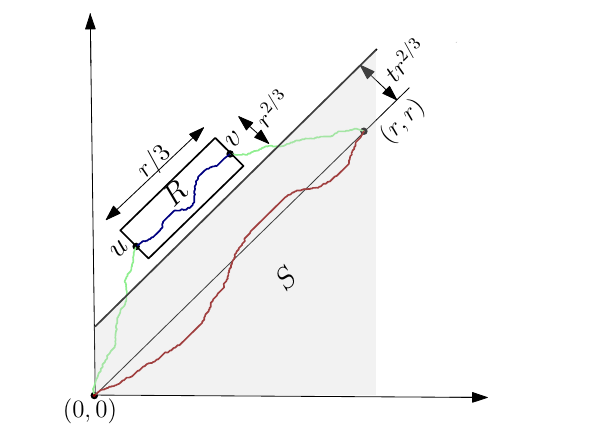}
        \caption{Illustration for the proof of Theorem \ref{thm2}. The blue path corresponds to event $A$, and the green paths connect the good blue path to the origin and the point $\bfr$. Events $B$ and $C$ ensure that the green paths are not too bad. On $D$, any path contained in $S$ will not have a 
        very large passage time. Thus on $A \cap B\cap C \cap D$, the geodesic from $\bfo$ to $\bfr$ must exit the region $S$.}
        \label{fig:2}
    \end{figure}
	Consider a rectangle $R$ obtained by translating the rectangle described in Lemma \ref{lemma6}. We specify the midpoints of the shorter sides to be
	$$u := \br{\frac{r}{3} - (t+1)r^{2/3}, \frac{r}{3} + (t+1)r^{2/3}} \text{     and        } v:= \br{\frac{2r}{3} - (t+1)r^{2/3}, \frac{2r}{3} + (t+1)r^{2/3}}.$$
	We define the set
    $$S= \bc{v : \psi\br{v} \leq tr^{2/3}, v \in \bZ^2}.$$ By our construction, the rectangle $R$ and set $S$ do not intersect. We define the event $A$ as we did for Lemma \ref{lemma6}. Furthermore, we define the events
	\begin{align*}
		B &:= \bc{T(\bfo, u) \geq  f(\bfo, u) - t r^{1/3}},\\ 
		C &:= \bc{T(v, \bfr) \geq  f(v,\bfr) - t r^{1/3}}.
	\end{align*}
	 Let $D$ be the event that all paths from $\bfo$ to $\bfr$ that stay inside the region $S$ have passage time less than $4r + t^2 r^{1/3}$.
	Clearly, 
	$$\P{D} \geq \P{T(\bfo, \bfr) \leq 4r + t^2 r^{1/3}}.$$
	By Theorem \ref{tempfl}, following \eqref{2} we have that for all $\epsilon >0$ and $t \leq (1-\epsilon)r^{1/3}$
	\begin{align}
        \label{3}
		\P{B}, \P{C} \geq 1 - Ce^{-ct^3}.
	\end{align}
        Similarly by \eqref{1}, we have that
        \begin{align}
        \P{D} \geq 1 - \ce{t}.
            \label{4}
        \end{align}        
	   On the event $A\cap B \cap C,$ the geodesic from $\bfo$ to $u$, the path in $R$ from Lemma \ref{lemma6} and the geodesic from $v$ to $\bfr$ form a path with passage time
	\begin{align*}
	&\geq 4r/3 + 10t^2r^{1/3} + f(\bfo, u) + f(v, \bfr) - 2t r^{1/3} \\
	&\geq 4r/3 + 10t^2r^{1/3} + 8r/3 -6(t+1)^2r^{1/3} - 2t r^{1/3}- t^2 r^{1/3} \\
	&\geq 4r + 2t^2 r^{1/3}
	\end{align*}
	for $t$ large enough, by using Taylor expansion of $f$ and dominating the error terms by $t^2r^{1/3}$. 
 
	Thus, on the event $A \cap B \cap C \cap D$, we have a path from $\bfo$ to $\bfr$ that is not contained in $S$, which has a higher passage time than any path in $S$. This prevents the geodesic $\Gamma_r$ from being contained in the region $S$. Thus, there exists a point $w \in \Gamma_r$ such that $\psi\br{w} > tr^{2/3}$ causing $\TF^r(r) \geq tr^{2/3}$. Furthermore, the event $A$ is independent of the event $B \cap C \cap D$. Thus by Lemma \ref{lemma6}, \eqref{3} and \eqref{4}, we have that
	$$ \P{\TF^r(r) \geq tr^{2/3}} \geq \P{A \cap B\cap C \cap D} \geq Ce^{-ct^3} \br{1-3Ce^{-ct^3}}.$$
	Now, by taking $t$ to be large enough, we have that
	$$\P{\TF^r(r) \geq tr^{2/3}} \geq Ce^{-ct^3}.$$
\end{proof}
\begin{remark}
    \label{rmk: t2}
    The proof of Theorem \ref{thm2} uses a simplified version of the arguments made in \cite[Proposition 1.4]{hammond2018modulus}. Instead of dealing with both cases of High and Low as in \cite[Proposition 1.4]{hammond2018modulus}, we showed that case High occurs by Lemma \ref{lemma6} and modified the arguments accordingly for exponential LPP.
\end{remark}
We are now in a position to present the proof of Theorem \ref{nproof}. We shall use the arguments we made in Theorem \ref{thm2} along with the ordering of geodesics to prove our result.
\begin{proof}[Proof of Theorem \ref{nproof}]
    \begin{figure}
        \begin{subfigure}[b]{0.45 \textwidth}
            \includegraphics[scale = 0.8]{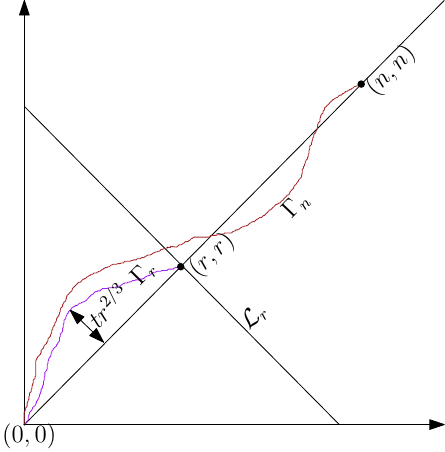}
            \label{fig:3}
            \caption{$\Gamma_r$ forcing $\Gamma_n$ to have large fluctuation}
        \end{subfigure}
        \hfill
        \begin{subfigure}[b]{0.45 \textwidth}
            \includegraphics[scale = 0.8]{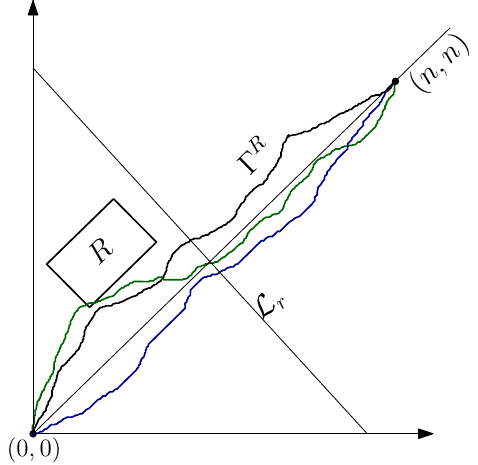}
            \caption{The event $E^R$ implying event $E$}
            \label{fig:4}
        \end{subfigure}
        \caption{Illustrations for Theorem \ref{nproof}:(A) Since two geodesics cannot intersect twice, the condition that $\Gamma_n$ has non-negative fluctuation at $\cL_r$ forces it to be above $\Gamma_r$. (B) If the event $E$ does not happen while $E'$ does, then $\Gamma_n$ must look like either the green or blue paths. The maximality of $\Gamma^R$ implies that the blue path cannot be a geodesic. Meanwhile, suppose the green path is a geodesic. In that case, the path we get from following $\Gamma^R$ initially and then shifting to the green path after, the last intersection will contradict the maximality of $\Gamma^R$.}
        \hfill
    \end{figure}
    As in the proof of Theorem \ref{thm2}, define the sets $A,B,C,D$ and the rectangle $R$. Thus, as previously shown, on the event $A \cap B \cap C \cap D$, we will have that $\TF^r(r) \geq tr^{2/3}$. In fact, we have shown that there exists a point $w \in \Gamma_r$ such that $$\psi\br{w} > tr^{2/3}.$$
    
    We shall now use $\Gamma_n(r)$ to denote the intersection point of the geodesic $\Gamma_n$ and the line $\cL_r$. We say that the geodesic has negative fluctuation at $\cL_r$ if $\psi\br{\Gamma_n(r)} < 0$ and call the fluctuation non-negative otherwise. Define the event $E= \bc{\psi\br{\Gamma_n(r)} \geq 0}$. Now on the event $A \cap B \cap C \cap D \cap E$, $\Gamma_n$ will be forced to be to the left of $\Gamma_r$, since any two geodesics can not intersect twice. Thus, the point $w$ on $\Gamma_r$ will force $\TF^n(r) \geq tr^{2/3}$.
    \begin{align*}
        \P{A\cap B \cap C \cap D \cap E} &= \P{A} \br{\P{B \cap C \cap D \cap E \:\vert A}}
        \\ &\geq \P{A} \br{\P{E \vert A} - \P{\br{B \cap C \cap D}^c \vert A}}
        \\ &= \P{A} \br{\P{E \vert A} - \P{\br{B \cap C \cap D}^c}}.
    \end{align*}
    Our aim now will be to lower bound the term $\P{E \vert A}$. Let $T^R(\gamma)$ be the passage time for the upright path $\gamma$ to go from $\bfo$ to $\bfn$ after discounting the weights in $R$. Therefore,
    $$T^R(\gamma) = \sum_{u \in \gamma, u \notin R} X_u.$$
    We use $\Gamma^R$ to denote the path from $\bfo$ to $\bfn$ with maximal $T^R$. Since our random variables are continuous, such a path is almost surely unique. Let $\Gamma^R(r)$ denote the intersection point of $\cL_r$ and $\Gamma^R$. We define $$E^R = \bc{\Gamma^R \text{ lies below } R \text{ and } \psi\br{\Gamma^R(r)} \geq 0 \text{ at } \cL_r}.$$ 
    Notice that the event $E^R \subseteq E$, as the path $\Gamma^R$ forces the path $\Gamma_n$ to have non-negative fluctuation at $\cL_r$. This is because increasing the weights from 0 in $R$ can only increase the transversal fluctuation of the geodesic at $\cL_r$.
    
    At the same time, on the event 
    \[
    E' = \bc{\Gamma_n \text{ lies below } R \text{ and }\psi\br{\Gamma_n(r)} \geq 0},
    \]
    the path $\Gamma_n$ and $\Gamma_R$ will be the same. Thus $E' \subseteq E^R$ and  for large enough $t$, by Theorem \ref{transfl}
    $$\P{E^R} \geq \P{E'} \geq \frac{1}{2} - \ce{t} \geq \frac{1}{4}.$$
    
    $E^R$ is also independent of the event $A$ since $E^R$ does not depend on the passage time of the vertices inside $R$ while $A$ depends only on the weights of those vertices. Thus, we have that
    \begin{align*}
        \P{A\cap B \cap C \cap D \cap E} &\geq \P{A} \br{\P{E^R \vert A} - \P{\br{B \cap C \cap D}^c}} \\
        &=\P{A}\br{\P{E^R} - \P{\br{B \cap C \cap D}^c}}.
    \end{align*}
    Now by \eqref{3}, \eqref{4} and Lemma \ref{lemma6}, we can claim that
    \begin{align*}
        \P{A}\br{\P{E^R} - \P{\br{B \cap C \cap D}^c}} &\geq \ce{t} \br{1/4 - 3\ce{t}} \\
        & \geq \ce{t}.
    \end{align*}
\end{proof}
\begin{remark}
    The result directly puts us in a position to comment on the fluctuations of semi-infinite geodesics. Theorem \ref{nproof} will give us a similar lower bound for large fluctuations of the semi-infinite geodesics in the direction of $x=y$ since the finite geodesics converge subsequentially to the semi-infinite geodesic. Such a lower bound for semi-infinite geodesics was proven in \cite[Theorem 2.8]{Timo}. Our result also holds for finite geodesics while their result hold for all semi-infinite geodesics in non-axial directions.
\end{remark}
\begin{remark}
    The proof uses the event $E^R$, which is independent of the rectangle $R$ and the symmetry of the problem. While considering other non-axial directions, one can still produce similar constructions for the event $E^R$, but the symmetry is lost. However, it is to be noted that the symmetry is used to claim that the $\P{E^R}$ is greater than some constant; this should be achievable even in the non-axial directions. It is also to be noted that we did not use any explicit property of exponential LPP. All our arguments follow from the moderate deviation result from integrable probability. Thus, the same arguments should hold for other solvable models like the Poissonian LPP and Geometric LPP. 
\end{remark}
\begin{remark}
    We have proven a lower bound for the maximal transversal fluctuation of the geodesic up to the line $\cL_r$. The same lower bound (up to exponents) should also hold for large transversal fluctuation of the geodesic at the line $\cL_r$. However, this requires some refinement of the arguments we make.
\end{remark}
\subsection*{Acknowledgement} 
I thank my supervisor, Riddhipratim Basu, for the useful discussions and comments.

\bibliographystyle{alpha}
\bibliography{bib}
\end{document}